\documentclass[14pt]{amsart}

\usepackage{amsmath, amsthm, amssymb, amsfonts}
\usepackage{thmtools}
\usepackage[utf8]{inputenc}
\usepackage[english]{babel}
\usepackage{tikz-cd}
\usepackage{comment}
\usepackage[normalem]{ulem}
\usepackage[%
 setpagesize=false,%
 bookmarks=true,%
 bookmarksdepth=tocdepth,%
 bookmarksnumbered=true,%
 colorlinks=true,%
 pdftitle={},%
 pdfsubject={},%
 pdfauthor={},%
 pdfkeywords={}%
]{hyperref}
\usepackage{cleveref}

\newcommand{\ulPC}{\underline{P}C}

\newcommand{\ulPHH}{\underline{P}HH}
\newcommand{\calulPHH}{\mathcal{\underline{P}HH}}
\newcommand{\ulMHH}{\underline{M}HH}
\newcommand{\calulMHH}{\mathcal{\underline{M}HH}}
\newcommand{\ulMO}{\underline{M}\mathcal{O}}
\newcommand{\ulPOmega}{\underline{P}\Omega}
\newcommand{\ulMOmega}{\underline{M}\Omega}

\newcommand{\ulMHC}{\underline{M}HC}
\newcommand{\ulMDM}{\underline{\mathbf{M}}\mathbf{DM}}
\newcommand{\DMeff}{\mathbf{DM}^{\mathrm{eff}}}
\newcommand{\ulMDMeff}{\ulMDM^{\mathrm{eff}}}
\newcommand{\bfulMOmega}{\underline{\mathbf{M}}\mathbf{\Omega}}

\newcommand{\ulMHP}{\underline{M}HP}

\newcommand{\ulMHNC}{\underline{M}HC^{-}}
\newcommand{\bfulMO}{\mathbf{\underline{M}\calO}}
\newcommand{\rep}{\mathrm{rep}}
\newcommand{\gp}{\mathrm{gp}}
\newcommand{\cyc}{\mathrm{cyc}}

\newcommand{\Zar}{\mathrm{Zar}}
\newcommand{\op}{\mathrm{op}}

\newcommand{\Spec}{\mathrm{Spec}}

\newcommand{\bbA}{\mathbb{A}}

\newcommand{\bbH}{\mathbb{H}}

\newcommand{\bbL}{\mathbb{L}}

\newcommand{\bbN}{\mathbb{N}}

\newcommand{\bbP}{\mathbb{P}}
\newcommand{\bbQ}{\mathbb{Q}}

\newcommand{\bbT}{\mathbb{T}}

\newcommand{\bbZ}{\mathbb{Z}}

\newcommand{\calA}{\mathcal{A}}
\newcommand{\calB}{\mathcal{B}}

\newcommand{\calD}{\mathcal{D}}

\newcommand{\calI}{\mathcal{I}}

\newcommand{\calM}{\mathcal{M}}

\newcommand{\calO}{\mathcal{O}}

\newcommand{\calX}{\mathcal{X}}

\newcommand{\overX}{\overline{X}}
\newcommand{\overY}{\overline{Y}}

\numberwithin{equation}{section}

\theoremstyle{plain}
\newtheorem{thm}{Theorem}[section]
\crefname{thm}{Theorem}{Theorems}
\newtheorem{dfn}[thm]{Definition}
\crefname{dfn}{Definition}{Definitions}
\newtheorem{prop}[thm]{Proposition}
\crefname{prop}{Proposition}{Propositions}
\newtheorem{lemma}[thm]{Lemma}
\crefname{lemma}{Lemma}{Lemmas}
\newtheorem{cor}[thm]{Corollarry}
\crefname{cor}{Corollarry}{Corollarries}
\theoremstyle{remark}
\newtheorem{rmk}[thm]{Remark}
\crefname{rmk}{Remark}{Remarks}

\theoremstyle{definition}

\crefname{ex}{Example}{Examples}

\begin{document}

\title{Hochschild and cyclic Homologies with bounded poles}
\author{Masaya Sato}
\address{Tokyo Institute of Technology, 
2-12-1 Ookayama, Meguro-ku, 
Tokyo 152-8551, Japan}
\email{sato.m.bj@m.titech.ac.jp}

\maketitle
\begin{abstract}
     We show that the classical Hochschild homology and (periodic and negative) cyclic homology groups are representable in the category of motives with modulus. We do this by constructing Hochschild homology and (periodic and negative) cyclic homologies for modulus pairs. We show a modulus version of HKR theorem, that is, there exists an isomorphism between modulus Hochschild homology and modulus K\"ahler differentials for affine normal crossing modulus pairs. By using the representability of modulus Hodge sheaves in the category of motives with modulus, we construct an object of the category of motives with modulus which represents modulus Hochschild homology. Similarly, We compare modulus de Rham cohomology and modulus cyclic homologies and obtain a representability of modulus cyclic homologies.
\end{abstract}

\tableofcontents

\section{Introduction}\label{intro}

Let $A$ be a commutative ring. Then there is a trace map from algebraic $K$-theory $K(A)$ to cyclic homology $HC(A)$, and this map induces an isomorphism
\[
K_n(A,I)\otimes\bbQ\cong HC_{n-1}(A,I)\otimes\bbQ
\]
for nilpotent ideal $I$ \cite{Goo86}. Therefore cyclic homology is important to compute algebraic $K$-theory. Cyclic homology is constructed from Hochschild homology $HH(A)$ as a homotopy orbits $HH(A)_{h\bbT}$. In this paper, we show Hochschild and (negative, periodic) cyclic homologies are representable in the category of motives with modulus $\ulMDMeff_k$. \\
\indent Voevodsky constructed the triangulated category of motives $\DMeff_k$ over a field $k$ which unifies $\bbA^1$-invariant cohomology theories. His theory has many applications, but it can not be applied to Hochschild and cyclic homologies, because they are not $\bbA^1$-invariant. In \cite{KMSY21a}, \cite{KMSY21b}, \cite{KMSY22}, Kahn, Miyazaki, Saito, and Yamazaki constructed a triangulated category $\ulMDMeff_k$ of motives with modulus which contains $\mathbf{DM}^{\mathrm{eff}}_k$ as a full subcategory. Just as $\DMeff_k$ is constructed from smooth schemes $X$ over $k$, $\ulMDMeff$ is constructed from \emph{modulus pairs} $\calX=(\overX,X^\infty)$, where $\overX$ is a separated and of finite type $k$-scheme, and $X^\infty$ is a effective Cartier divisor of $\overX$ such that $\overX\setminus |X^\infty|$ is smooth over $k$. When $k$ admits resolution of singularities, e.g. $\mathrm{char}(k)=0$ \cite{Hir64}, $\ulMDMeff_k$ can be constructed from \emph{strict normal crossing} modulus pairs $\calX=(\overX,X^\infty)$, that is, $\overX$ is smooth and $|X^\infty|$ has strict normal crossing.\footnote{cf.\cite[Corollary 1.9.5]{KMSY21a} By the construction of $\ulMDMeff_k$, abstract blow-ups become isomorphisms.} An \emph{ambient morphism} $f:(\overX,X^\infty)\to(\overY,Y^\infty)$ of modulus pairs is a morphism of schemes $f:\overX\to\overY$ such that $X^\infty\geq f^*Y^\infty$. The modulus pair $\overline{\Box}:=(\bbP^1,\infty)$ is called the \emph{cube} and plays a role of unit interval instead of $\bbA^1$.\\
\indent In \cite{KM23a}, \cite{KM23b}, Kelly and Miyazaki constructed objects $\bfulMO$ and $\bfulMOmega^q$ in $\ulMDMeff_k$ such that they represents Zariski and Hodge cohomologies, if $\mathrm{char}(k)=0$. Koizumi did that if $k$ admits resolutions of singularities \cite{Koi23}. More precisely:
\begin{thm}\label{Hodge realization}
For any strict normal crossing modulus pair $\calX=(\overX,X^\infty)$, there are isomorphisms
\begin{align*}
    &\mathrm{Hom}_{\ulMDMeff_k}(M(\calX),\bfulMO[n])\cong H^n_{\mathrm{Zar}}(\overX,\ulMO_{\calX})\\
    &\mathrm{Hom}_{\ulMDMeff_k}(M(\calX),\bfulMOmega^q[n])\cong H^n_{\mathrm{Zar}}(\overX,\ulMOmega^q_{\calX}),
\end{align*}   
where $\ulMO_{\calX}:=\sqrt{\calI}\otimes\calI^{-1}$, $\calI$ is the ideal sheaf corresponds to the effective Cartier divisor $X^\infty$, and $\ulMOmega^q_{\calX}:=\ulMO_{\calX}\otimes\Omega^q_{\overX}(\log X^\infty)$.
\end{thm}
\indent If $A$ is a smooth $k$-algebra, then Hochschild homology $HH_n(A/k)$ is isomorphic to the module of K\"ahler differentials $\Omega^n_{A/k}$, via the so-called HKR theorem, \cite{HKR62}. Furthermore if $k$ contains $\bbQ$,
\[
HC_n(A/k)\cong (\Omega^n_{A/k}/d\Omega^{n-1}_{A/k})\times H^{n-2}(\Omega^*_{A/k})\times H^{n-4}(\Omega^*_{A/k})\times\cdots
\]
by \cite[Theorem 3.4.12]{Lod98}. In this paper, we define Hochschild homology $\ulMHH(A,f)$ for affine modulus pair $(\Spec A,(f))$ and show the modulus version of the HKR theorem.

\begin{thm}[\cref{HKR}]
    Let $(\mathrm{Spec}A,(f))$ be a strict normal crossing modulus pair such that $f$ is a product of prime elements. Then there is an isomorphism
     \begin{align*}
         \uline{M}\Omega^*(A,f)\cong\uline{M}HH_*(A,f).
     \end{align*}
\end{thm}

Furthermore, we define (negative, periodic) cyclic homologies for modulus pairs $\ulMHC(A,f)$, $\ulMHNC(A,f)$, $\ulMHP(A,f)$ and show isomorphims if $\mathrm{char}(k)=0$.

\begin{thm}[\cref{isom for cyclic}]
    If $(A,f)$ is strict normal crossing and $f=f_1^{r_1}\cdots f_n^{r_n}$ for prime elements $f_j\in A$, then
    \begin{align*}
        &\ulMHC_n(A,f)\cong (\ulMOmega^n_{(A,f)}/d\ulMOmega^{n-1}_{(A,f)})\times H^{n-2}(\ulMOmega^*_{(A,f)})\times H^{n-4}(\ulMOmega^*_{(A,f)})\times\cdots\\        &\ulMHNC(A,f)\cong\mathrm{Ker}(\ulMOmega_{(A,f)}^n\to \ulMOmega_{(A,f)}^{n+1})\times H^{n+2}(\ulMOmega_{(A,f)^*})\times H^{n+4}(\ulMOmega_{(A,f)}^*)\times\cdots\\
        &\ulMHP_n(A,f)\cong \prod_{p\in\mathbb{Z}}H^{2p-n}(\ulMOmega^*(A,f)).
    \end{align*}
\end{thm}

Finally, we globalize this result and obtain realizations of Hochschild and (negative, periodic) cyclic homologies.

\begin{thm}[\cref{realization theorem}]\label{realization intro}
     For every strict normal crossing modulus pair $\mathcal{X}=(\overline{X},X^\infty)$, we have the following isomorphisms
     \begin{align*}
     &\mathrm{Hom}_{\ulMDMeff_k}(M(\mathcal{X}),\prod_{p\in\mathbb{Z}}\bfulMOmega^p[p-n])\cong\prod_{p\in\mathbb{Z}}H^{p-n}_{\Zar}(\overX,\ulMOmega^p_{\calX})\cong\ulMHH_n(\mathcal{X})\\
     &\mathrm{Hom}_{\ulMDMeff_k}(M(\mathcal{X}),\prod_{p\in\mathbb{Z}}\bfulMOmega^{\leq p}[2p-n])\cong\prod_{p\in\bbZ}\bbH^{2p-n}_{\Zar}(\overX,\ulMOmega^{\leq p}_{\calX})\cong\ulMHC_n(\mathcal{X})\\
     &\mathrm{Hom}_{\ulMDMeff_k}(M(\mathcal{X}),\prod_{p\in\mathbb{Z}}\bfulMOmega^{\geq p}[2p-n])\cong\prod_{p\in\bbZ}\bbH^{2p-n}_{\Zar}(\overX,\ulMOmega^{\geq p}_{\calX})\cong\ulMHNC_n(\mathcal{X})\\
     &\mathrm{Hom}_{\ulMDMeff_k}(M(\mathcal{X}),\prod_{p\in\mathbb{Z}}\bfulMOmega^\bullet[2p-n])\cong\prod_{p\in\bbZ}\bbH^{2p-n}_{\Zar}(\overX,\ulMOmega^{\bullet}_{\calX})\cong\ulMHP_n(\mathcal{X}).
     \end{align*}
 \end{thm}
 As a corollary, they satisfy blow-up invariance, and $\overline{\Box}$-invariance. Furthermore, if the divisor is empty    , we can recover the classical groups.

 \begin{cor}
     For any smooth, separated and of finite type $k$-scheme $X$, by taking $\calX=(X,\varnothing)$ in \cref{realization intro}, we have the following isomorphisms
     \begin{align*}
     &\mathrm{Hom}_{\ulMDMeff_k}(M(X,\varnothing),\prod_{p\in\mathbb{Z}}\bfulMOmega^p[p-n])\cong\prod_{p\in\mathbb{Z}}H^{p-n}_{\Zar}(X,\Omega^p_{X})\cong HH_n(X)\\
     &\mathrm{Hom}_{\ulMDMeff_k}(M(X,\varnothing),\prod_{p\in\mathbb{Z}}\bfulMOmega^{\leq p}[2p-n])\cong\prod_{p\in\bbZ}\bbH^{2p-n}_{\Zar}(X,\Omega^{\leq p}_{X})\cong HC_n(X)\\
     &\mathrm{Hom}_{\ulMDMeff_k}(M(X,\varnothing),\prod_{p\in\mathbb{Z}}\bfulMOmega^{\geq p}[2p-n])\cong\prod_{p\in\bbZ}\bbH^{2p-n}_{\Zar}(X,\Omega^{\geq p}_{X})\cong HC^-_n(X)\\
     &\mathrm{Hom}_{\ulMDMeff_k}(M(X,\varnothing),\prod_{p\in\mathbb{Z}}\bfulMOmega^\bullet[2p-n])\cong\prod_{p\in\bbZ}\bbH^{2p-n}_{\Zar}(X,\Omega^{\bullet}_{X})\cong HP_n(X).
     \end{align*}
 \end{cor}

\section{Acknowledgement}

The author would like to thank Shane Kelly for proposing and supporting this research.

\section{Modulus de Rham Complex}\label{deRham}

We will construct the modulus de Rham complex $\bfulMOmega^\bullet$ in the category of motives with modulus $\ulMDMeff_k$. If $\calX=(\Spec A,(f))$, then global sections of $\ulMO_\calX$ and $\ulMOmega^q_\calX$ are the following sub-$A$-modules of $A[f^{-1}]$ and $\Omega^q_{A[f^{-1}]}$.

\begin{dfn}[{\cite[Definition 3.1]{KM23a}, \cite[Definition 3.1]{KM23b}}]
    For a $k$-algebra $A$, and a nonzero divisor $f\in A$, define
    \begin{align*}    \ulMO(A,f):=\frac{\sqrt{(f)}}{f}\subseteq A[f^{-1}],
    \end{align*}
    $\ulPOmega^*_{(A,f)}$ to be the smallest sub-dg-algebra of $\Omega^*_{A[f^{-1}]}$ containing $A=\Omega^0_{A/k}$ and $d\log a:=\frac{da}{a}$ for $a\in A\cap A[f^{-1}]^*$ and
    \[    \ulMOmega^q:=\ulMO(A,f)\cdot\ulPOmega^q_{(A,f)}\subseteq \Omega^q_{A[f^{-1}]}.
    \]
\end{dfn}

We will construct the de Rham differential $d:\bfulMOmega^q_\calX \to \bfulMOmega^{q+1}_\calX$.

\begin{lemma}
    If $f=f_1^{r_1}\cdots f_n^{r_n}$ is a product of prime elements, then the differential of $\Omega^\bullet_{A[f^{-1}]}$ induces the differential on $\ulMOmega^\bullet_{(A,f)}$.
\end{lemma}

\begin{proof}
    Since $\ulMOmega^q_{(A,f)}=\ulMO(A,f)\cdot\ulPOmega^q_{(A,f)}$ and $\ulPOmega^\bullet_{(A,f)}$ is a complex, it is sufficient to show $d\ulMO(A,f)\subseteq \ulMOmega^1_{(A,f)}$. Since $\sqrt{f}=(f_1\cdots f_n)A$, for $\frac{a}{f}\in\ulMO(A,f)$, we have $a=a'g$ where $g=f_1\cdots f_n$. Thus
    \begin{align*}
    d\left(\frac{a}{f}\right)&=\frac{da}{f}-\frac{a}{f}d\log f\\
    &=\frac{g}{f}da'+\frac{a'}{f}dg-\frac{a}{f}d\log f\\
    &=\frac{g}{f}da'+\frac{a}{f}d\log g-\frac{a}{f}d\log f\in \ulMOmega^1_{(A,f)}.
    \end{align*}
\end{proof}
If $\mathcal{X}=(\overline{X},X^\infty)$ is a normal crossing modulus pair, then $\overline{X}$ is smooth, in particular locally factorial, so we can define the differential $d:\ulMOmega^*_{\calX}\to\ulMOmega^{*+1}_{\calX}$ by the following lemma.

\begin{lemma}\label{lem:factorization}
     Let $\Spec A$ be a Noetherian locally factorial affine scheme and $f\in A$ be an element. For any point $\mathfrak{p}\in\mathrm{Spec}(A)$, there is a neighborhood  $\mathfrak{p}\in\mathrm{Spec}(A[g^{-1}])$ such that $f=f_1^{r_1}\cdots f_n^{r_n}$ for some prime elements $f_j\in A[g^{-1}]$.
 \end{lemma}

 \begin{proof}
     Since $A_\mathfrak{p}$ is UFD, $f=f_1^{r_1}\cdots f_n^{r_n}$ for some prime elements $f_j\in A_\mathfrak{p}$. Let $\mathfrak{q}_j=(u_{j1},\dots,u_{jk}) \subseteq \mathfrak{p}$ be the prime ideal in $A$ such that $\mathfrak{q}_jA_\mathfrak{p}=f_jA_\mathfrak{p}$. Thus $u_{jl}=f_jv_{jl}$ for some $v_{jl}\in A_\mathfrak{p}$. We can write $f_j=s_j/g$, $v_{jl}=v_{jl}'/g$, and $f_j=\sum u_{jl}(w_{jl}/g)$ for some $s_j, v'_{jl}, w_{jl}\in A$ and $g\notin \mathfrak{p}$. Therefore $f=f_1^{r_1}\cdots f_n^{r_n}$ for $f_j=s_j/g\in A[g^{-1}]$ and each $f_j$ is a prime element because $\mathfrak{q}_jA[g^{-1}]=f_jA[g^{-1}]$.
 \end{proof}

\begin{prop}
    For every integer $q\geq0$, the differential
    \[
    d:\ulMOmega^q\to\ulMOmega^{q+1}
    \]
    is a morphism of sheaves with transfers.
\end{prop}

\begin{proof}
We need to show that the back square in the diagram
    \[
\begin{tikzcd}[row sep=tiny, column sep=tiny]
\ulMOmega^q(\mathcal{Y})\arrow[dd]\arrow[rr]\arrow[rd]&
&
\ulMOmega^{q+1}(\mathcal{Y})\arrow[dd]\arrow[rd]
\\
&
\Omega^{q}(Y^\circ)\arrow[rr,crossing over]&
&
\Omega^{q+1}(Y^\circ)\arrow[dd]
\\
\ulMOmega^q(\mathcal{X})\arrow[rr]\arrow[rd]&
&
\ulMOmega^{q+1}(\mathcal{X})\arrow[rd]
\\
&
\Omega^q(X^\circ)\arrow[rr]\arrow[uu,crossing over,leftarrow]&
&
\Omega^{q+1}(X^\circ).
\end{tikzcd}
\]
commutes for every modulus correspondence $W\in\underline{\mathbf{M}}\mathbf{Cor}(\mathcal{X},\mathcal{Y})$ for every normal crossing modulus pairs $\mathcal{X}=(\overline{X},X^\infty)$ and $\mathcal{Y}=(\overline{Y},Y^\infty)$, where $X^\circ=\overline{X}-X^\infty$, $Y^\circ=\overline{Y}-Y^\infty$. The top and the bottom squares commute by the definition of the differential on $\ulMOmega^*$. The left and the right squares commute by \cite[Theorem 4.3.]{KM23b}. Since $\ulMOmega^*(\mathcal{X})\to\Omega^*(X^\circ)$ and $\ulMOmega^*(\mathcal{Y})\to\Omega^*(Y^\circ)$ are injective, it is sufficient to show that the front square commutes. This case is checked in \cite[Theorem A.4.1]{KSYR}, and \cite{LW09}.
\end{proof}

\begin{thm}\label{de Rham realization}
    There exists a unique object $\underline{\mathbf{M}}\mathbf{\Omega}^\bullet\in\ulMDM^{\mathrm{eff}}_k$ such that for every strict normal crossing modulus pair $\mathcal{X}=(\overline{X},X^\infty)$ we have
    \[
    \mathrm{Hom}_{\ulMDMeff_k}(M(\mathcal{X}),\underline{\mathbf{M}}\mathbf{\Omega}^\bullet[n])\cong \mathbb{H}^n_{\mathrm{Zar}}(\overline{X},\ulMOmega^\bullet_{\mathcal{X}}). 
    \]
\end{thm}

\begin{proof}
    The Hodge-to-de Rham spectral sequence
    \[E_1^{p,q}=H^q_{\mathrm{Zar}}(\overX,\ulMOmega^p_{\calX})\implies\mathbb{H}^{p+q}_{\Zar}(\calX,\ulMOmega^\bullet)
    \]
    and the spectral sequence of extensions 
    \[
    E_1^{p,q}=\mathrm{Hom}_{\ulMDMeff_k}(M(\mathcal{X}),\underline{\mathbf{M}}\mathbf{\Omega}^p[q])\implies \mathrm{Hom}_{\ulMDMeff_k}(M(\mathcal{X}),\underline{\mathbf{M}}\mathbf{\Omega}^\bullet[p+q])
    \]
    induce the isomorphism 
    \[
    \mathrm{Hom}_{\ulMDMeff_k}(M(\calX),\bfulMOmega^\bullet[n])\cong\mathbb{H}^n_{\mathrm{Zar}}(\overX,\ulMOmega^\bullet_{\calX})
    \]
    by \cref{Hodge realization}.
\end{proof}

\section{Modulus Hochschild Homology}

 Recall that the Hochschild homology $HH_*(A)$ of an $k$-algebra $A$ is the homology of the simplicial ring $HH(A)$ whose $n$th ring is $A^{\otimes n+1}$ and face operators $d_i:A^{\otimes n+1}\to A^{\otimes n+2}$ and degeneracy operators are 
 \begin{align*}
     d_i(a_0\otimes\cdots a_n)&=
     \begin{cases}
    a_0\otimes\cdots\otimes a_ia_{i+1}\otimes\cdots a_n       & \quad \text{if } 0\leq i\leq n-1\\
    a_na_0\otimes\cdots\otimes a_{n-1}  & \quad \text{if } i=n
  \end{cases}\\
  s_i(a_0\otimes\cdots a_n)&=a_0\otimes\cdots a_i\otimes 1\otimes a_{i+1}\otimes\cdots\otimes a_n.
 \end{align*}
 Let $b$ be its differential
 \begin{align*}
     b(a_0\otimes\cdots\otimes a_n)=\sum_{i=0}^{n-1}(-1)^ia_0\otimes\cdots\otimes a_ia_{i+1}\cdots\otimes a_n+(-1)^na_na_0\otimes\cdots\otimes a_{n-1}.
 \end{align*}
 The complex $HH(A)$ has a structure of a graded commutative dg-algebra by the shuffle product $\nabla$. Furthermore, if $A$ is a smooth $k$-algebra, then we have an HKR isomorphism of graded algebras
 \begin{align*}
     \Omega^*_{A/k}\xrightarrow{\sim}HH_*(A).
 \end{align*}
  In this section, we construct the modulus Hochschild homology $\uline{M}HH_q$ which is a Hochschild homology for modulus pairs and show the modulus HKR theorem.\\
  \indent We fix a field $k$.
 \begin{dfn}
     For a $k$-algebra $A$ and a nonzero divisor $f\in A$, we write $\ulPHH(A,f)$ for the sub-simplicial ring of $HH(A[f^{-1}])$ whose $n$th ring is the ring of sums of elements of the form
     \begin{align*}
         a_0m_0\otimes\cdots\otimes a_nm_n\in A[f^{-1}]^{\otimes n+1},
     \end{align*}
     where $a_i\in A,\ m_i\in A[f^{-1}]^*$ and $m_0\cdots m_n\in A$. Define sub-simplicial module
     \begin{align*}
         \ulMHH(A,f):=\uline{M}\mathcal{O}(A,f)\cdot\ulPHH(A,f)\subseteq HH(A[f^{-1}]),
     \end{align*}
     We call it \emph{modulus Hochschild homology}. Let $\uline{P}HH_*(A,f):=H_*(\ulPHH(A,f))$, $\ulMHH_*(A,f):=H_*(\ulMHH(A,f))$ to be the homology groups.
 \end{dfn}

 \begin{rmk}
     By definition, $\ulPHH(A,f)$ is a cdga by the shuffle product $\nabla$.
 \end{rmk}

 We will compare $\ulPHH$ with log Hochschild homology. Before proving that, we recall some definitions about pre-log rings.

 \begin{dfn}
     A pre-log ring $(A,M,\alpha)$ is a triple of a commutative ring $A$, a commutative monoid $M$, and a map of monoids $\alpha:M\to A$. A map of pre-log rings $(f,f^\flat):(A,M,\alpha)\to(B,N,\beta)$ is a pair of a map of rings $f:A\to B$ and a map of monoids $f^\flat:M\to N$, such that 
     \[
     \begin{tikzcd}
         M \ar[r,"f^\flat"] \arrow[d,"\alpha"] & N \ar[d,"\beta"] \\
         A \ar[r,"f"] & B
     \end{tikzcd}
     \]
     is commutative. For a surjective map of monoids $f:M\to N$, \emph{repletion} of $f$ is the pullback $M^\rep:=N\times_{N^\gp}M^\gp$. For any map $f:M\to N$, let $(\oplus_M^nN)^\rep$ be the repletion of $\oplus_M^nN\to N$.
 \end{dfn}

 \begin{prop}\label{comparison theorem of logHH}
     Let $(R,P)\to(A,M)$ be a map of pre-log rings. If there are factorizations $M=M_1\oplus M_2, P=P_1\oplus P_2$ such that $P\to M$ is the sum of the $P_i\to M_i$, $P_2$ and $M_2$ are groups, and $R\otimes_{\bbZ[P_1]}\bbZ[M_1]\to A$ is flat, then there is an equivalence\footnote{This equivalence holds for any map $(R,P)\to(A,M)$ in the derived sense \cite[Proposition 5.6]{logHH}.}
     \begin{align}\label{comparison}         A\otimes^{\mathbb{L}}_{(A\otimes_R A)\otimes_{\mathbb{Z}[M\oplus_PM]}\mathbb{Z}[(M\oplus_PM)^{\mathrm{rep}}]}A\simeq C((A,M)/(R,P)),
     \end{align}
     where $ C((A,M)/(R,P))$ is a complex such that $C_n((A,M)/(R,P))=A^{\otimes_R{n+1}}\otimes_{\mathbb{Z}[\oplus_P^{n+1}M]}\mathbb{Z}[(\oplus_P^{n+1}M)^{\mathrm{rep}}]$ with Hochschild differential.
 \end{prop}

 \begin{proof}
     Let $B_n:=A^{\otimes_Rn}\otimes_{\mathbb{Z}[\oplus_P^nM]}\mathbb{Z}[(\oplus_P^nM)^{\mathrm{rep}}]$. The Bar resolution
     \[
     B\to A=(\cdots\to B_3\to B_2)\to A
     \]
     is a contractible augmented simplicial object. If $B_{n+2}$ is flat over $B_2$, that is, $B$ is a $B_2$-flat resolution of $A$, and if $B_{n+2}\otimes_{B_2}A\cong B_{n+1}$, then
     \[     A\otimes^{\bbL}_{B_2}A\simeq B\otimes_{B_2}A\cong  C((A,M)/(R,P)).
     \] 
    Therefore it is sufficient to show that $B_{n+2}$ is flat over $B_2$, and $B_{n+2}\otimes_{B_2}A\cong B_{n+1}$. 

     Since $\bigoplus_P^nM=\left(\bigoplus_{P_1}^nM_1\right)\oplus\left(\bigoplus_{P_2}^nM_2\right) $, and since $P_2$ and $M_2$ are groups,
     \begin{align}\label{chart}
     \Big(\bigoplus_P^nM\Big)^{\mathrm{rep}}=\Big(\bigoplus_{P_1}^nM_1\Big)^{\mathrm{rep}}\oplus\Big(\bigoplus_{P_2}^nM_2\Big)^{\mathrm{rep}}=\Big(\bigoplus_{P_1}^nM_1\Big)^{\mathrm{rep}}\oplus\Big(\bigoplus_{P_2}^nM_2\Big).
     \end{align}
     Let $R':=R\otimes_{\bbZ[P_1]}\bbZ[M_1]$. Then
     \begin{align}\label{replete Bar construction}
         B_n\cong A^{\otimes_{R'}n}\otimes_{\mathbb{Z}[M_1]}\mathbb{Z}[(\oplus_{P_1}^nM_1)^{\mathrm{rep}}]
     \end{align}
     because
     \begin{align*}
         B_n&=A^{\otimes_Rn}\otimes_{\mathbb{Z}[\oplus_P^nM]}\mathbb{Z}[(\oplus_P^nM)^{\mathrm{rep}}]\\
         &\cong A^{\otimes_Rn}\otimes_{\mathbb{Z}[\oplus_{P_1}^nM_1]}\mathbb{Z}[(\oplus_{P_1}^nM_1)^{\mathrm{rep}}]\\
         &\cong A^{\otimes_{R'}n}\otimes_{R'}(R'^{\otimes_Rn})\otimes_{\mathbb{Z}[\oplus_{P_1}^nM_1]}\mathbb{Z}[(\oplus_{P_1}^nM_1)^{\mathrm{rep}}]\\
         &\cong A^{\otimes_{R'}n}\otimes_{\mathbb{Z}[M_1]}(\mathbb{Z}[M_1]^{\otimes_{\mathbb{Z}[P_1]}n})\otimes_{(\mathbb{Z}[M_1]^{\otimes_{\mathbb{Z}[P_1]}n})}\mathbb{Z}[(\oplus_{P_1}^nM_1)^{\mathrm{rep}}]\\
         &\cong  A^{\otimes_{R'}n}\otimes_{\mathbb{Z}[M_1]}\mathbb{Z}[(\oplus_{P_1}^nM_1)^{\mathrm{rep}}],
     \end{align*}
     where the third isomorphism is induced by $(-)\otimes_RR'=(-)\otimes_R(R\otimes_{\mathbb{Z}[P_1]}\mathbb{Z}[M_1])\cong(-)\otimes_{\mathbb{Z}[P_1]}\mathbb{Z}[M_1]$ and $\mathbb{Z}[\oplus_{P_1}^nM_1]\cong\mathbb{Z}[M_1]^{\otimes_{\mathbb{Z}[P_1]}n}$. Since 
     \[
     B_{n+2}\cong  A^{\otimes_{R'}n}\otimes_{R'}((A\otimes_{R'}A)\otimes_{\mathbb{Z}[M_1]}\mathbb{Z}[(M_1\oplus_{P_1}M_1)^{\mathrm{rep}}])\otimes_{\mathbb{Z}[(M_1\oplus_{P_1}M_1)^{\mathrm{rep}}]}\mathbb{Z}[(\oplus_{P_1}^nM_1)^{\mathrm{rep}}],
     \]
     for any $B_2$-module $N$
     \[
     B_{n+2}\otimes_{B_2}N\cong  A^{\otimes_{R'}n}\otimes_{R'}N\otimes_{\mathbb{Z}[(M_1\oplus_{P_1}M_1)^{\mathrm{rep}}]}\mathbb{Z}[(\oplus_{P_1}^{n+2}M_1)^{\mathrm{rep}}].
     \]
     Since $A$ is flat over $R'$, it is sufficient to show $\mathbb{Z}[(\oplus_{P_1}^{n+2}M_1)^{\mathrm{rep}}]$ is flat over $\mathbb{Z}[(M_1\oplus_{P_1}M_1)^{\mathrm{rep}}]$. However, for any homomorphism of monoids $P\to M$, $(\oplus_P^{n+1}M)^{\mathrm{rep}}\cong M\oplus (M^{\mathrm{gp}}/P^{\mathrm{gp}})^n$ (cf. \cite[Lemma 5.10.]{logHH}). Therefore
     \[
     N\otimes_{\mathbb{Z}[(M_1\oplus_{P_1}M_1)^{\mathrm{rep}}]}\mathbb{Z}[(\oplus_{P_1}^{n+2}M_1)^{\mathrm{rep}}]\cong N\otimes_{\bbZ[M_1]}\bbZ[M_1\oplus (M_1^{\mathrm{gp}})^n].
     \]
      Since $A$ is flat over $R'$ by assumption, $B_{n+2}$ is flat pver $B_n$. Furthermore, if $N=A$, that isomorphism says $B_{n+2}\otimes_{B_2}A\cong B_{n+1}$.
 \end{proof}
 
 The left hand side of \cref{comparison} is the definition of log Hochschild homology in \cite[Definition 5.3]{logHH}. We will check that if $(A,f)$ is strict normal crossing and $f=f_1^{r_1}\cdots f_n^{r_n}$ is a product of prime elements, then we have an isomorphism
 \begin{align}\label{comparison1}
     C((A,M)/(R,P))\xrightarrow{\sim} \ulPHH(A,f),
 \end{align}
 where $(A,M)=(A,A\cap A[f^{-1}]^*)$, $(R,P)=(k,k^*)$, and thus we have an equivalence
 \begin{align}\label{comparison2}
     A\otimes^{\mathbb{L}}_{(A\otimes_R A)\otimes_{\mathbb{Z}[M\oplus_PM]}\mathbb{Z}[(M\oplus_PM)^{\mathrm{rep}}]}A\simeq \ulPHH(A,f).
 \end{align}
 There is a canonical surjection of simplicial rings $C((A,M)/(R,P))\to \ulPHH(A,f)$. To show it is injective, it is sufficient to show both inject into $C(A[f^{-1}])$, but for the target this is the definition. By assumption, $M=A^*\times f_1^{\bbN}\times\cdots\times f_n^{\bbN}\cong \bbN^n\oplus A^*$ and $P=0\oplus k^*$. By \cref{replete Bar construction}, 
 \[
 B_{m+1}\cong A^{\otimes_{R'}m+1}\otimes_{\bbZ[M_1]}\bbZ[(M_1^m)^{\mathrm{rep}}]\cong A^{\otimes_{R'}m+1}\otimes_{R'}R'[(M_1^m)^{\mathrm{rep}}],
 \]
 where $M_1=\bbN^n$, $M_2=A^*$, $P_1=0$, $P_2=k^*$, and $R'=k[M_1]\cong k[x_1,\dots,x_n]$. If $A$ is flat over $R'=k[M_1]$, then  there is a injection $A^{\otimes_{R'}n+1}\otimes_{R'}R'[(M_1^n)^{\mathrm{rep}}]\to A^{\otimes_{R'}n+1}\otimes_{R'}R'[(M_1^n)^{\mathrm{gp}}]\cong A[f^{-1}]^{\otimes n+1}$. Thus to use \cref{comparison theorem of logHH},  it remains to show the flatness of $R'\to A$.

 \begin{lemma}\label{local coodinate}
     For any normal crossing modulus pair $(A,f)$ such that $f=f_1^{r_1}\cdots f_n^{r_n}$ for prime elements $f_1,\dots,f_n\in A$, there exists an \'etale covering $\{\mathrm{Spec}(A_i)\to\mathrm{Spec}(A)\}$ such that $R'=k[x_1,\dots,x_n]\to A\to A_i$ factors through $k[x_1,\dots,x_n]\to k[y_1,\dots,y_m]\to A_i$ where the first map is flat, and the second map is \'etale. In particular, $R'\to A$ is flat.
 \end{lemma}

 \begin{proof}
     First, we construct an \'etale covering $\{\Spec A_i\to\Spec A\}$ with \'etale morphisms $\Spec A_i\to \bbA^m$. By definition, there is an \'etale covering $\{\mathrm{Spec}(B)\to\mathrm{Spec}(A)\}$ with $(B,f)$ strict normal crossing. For each maximal ideal $\mathfrak{m}\in\mathrm{Spec}(B)$, there is a regular system of parameters $b_1,\dots b_m\in\mathfrak{m}B_{\mathfrak{m}}$ such that $f=b_1^{s_1}\cdots b_l^{s_l}$. Then $b_1,\dots b_m\in\mathfrak{m}-\mathfrak{m}^2$ are prime elements, and 
     \[     \Omega_{B_{\mathfrak{m}}/k}=\bigoplus_{i=1}^nB_{\mathfrak{m}}db_i.
     \]
     By \cref{lem:factorization}, there is a $g\notin \mathfrak{m}$ such that $b_1,\dots, b_n \in B[g^{-1}]$ are primes, $f=b_1^{s_1}\cdots b_l^{s_l}$, and 
     \[     \Omega_{B[g^{-1}]/k}=\bigoplus_{i=1}^nB[g^{-1}]db_i.
     \]
     Since $B':=B[g^{-1}]$ is finite type over $k$, there is a surjection $S:=k[y_1,\dots,y_m,z_1,\dots,z_c]\to B'$ with $y_i\mapsto b_i$. Let $I$ be its kernel. Since $B'$ is smooth over $k$, there are short exact sequences
     \[
        \begin{tikzcd}
            0 \ar[r] & I/I^2 \ar[r] \ar[d] & \Omega_{S/k}\otimes_SB' \ar[r] \ar[d] & \Omega_{B'/k} \ar[r] \ar[d] & 0\\
            0 \ar[r] & \bigoplus_{j=1}^cB'dz_j \ar[r] & (\bigoplus_{j=1}^cB'dz_j)\oplus(\bigoplus_{i=1}^mB'dy_i) \ar[r] & \bigoplus_{i=1}^mB'dy_i \ar[r] & 0,
        \end{tikzcd}
     \]
     where vertical maps are isomorphimsms. So there is a basis $h_1,\dots, h_c\in I/I^2$. By Nakayama's lemma (\cite[\href{https://stacks.math.columbia.edu/tag/00DV}{00DV}(3)]{stacks-project}), there is a $h\in 1+I$ such that $hI\subset (h_1,\dots ,h_c)$, $I[h^{-1}]=(h_1,\dots, h_c)[h^{-1}]$. Since $h=1$ in $B'$,
     \[
     B'\cong S/I\cong S[h^{-1}]/I[h^{-1}]\cong k[y_1,\dots,y_m,z_1,\dots,z_{c+1}]/(h_1,\dots,h_c,hz_{c+1}-1).
     \]
     Let $h_{c+1}:=hz_{c+1}-1$ and $J:=(h_1,\dots,h_{c+1})$. By \cite[\href{https://stacks.math.columbia.edu/tag/08JZ}{08JZ}(1)]{stacks-project}, $J/J^2$ is free with basis $h_1,\dots h_{c+1}$. By the short exact sequence 
     \[
     0\to J\to S':=k[y_1,\dots,y_m,z_1,\dots,z_{c+1}]\to B'\to 0
     \]
     we obtain short exact sequences and vertical isomorphisms again
     \[
        \begin{tikzcd}
            0 \ar[r] & J/J^2 \ar[r] \ar[d] & \Omega_{S'/k}\otimes_{S'}B' \ar[r] \ar[d] & \Omega_{B'/k} \ar[r] \ar[d] & 0\\
            0 \ar[r] & \bigoplus_{j=1}^{c+1}B'dz_i \ar[r] & (\bigoplus_{j=1}^{c+1}B'dz_j)\oplus(\bigoplus_{i=1}^mB'dy_i) \ar[r] & \bigoplus_{i=1}^mB'dy_i \ar[r] & 0.
        \end{tikzcd}
     \]
     Since left vertical map is a invertible matrix $(\partial h_i/\partial z_j)$, $B'\cong S'/J$ is a standard smooth [Stacks, 00T6]. Then $k[y_1,\dots,y_m]\to B'$ is also standard smooth and relative dimension $0$, so it is \'etale.\\
     \indent Next, we show the above construction induces a factorization of $R'=k[x_1,\dots,x_n]\to A\to B'$. By \cite[Lemma 3.5]{KM23a}, $(f_1\cdots f_n)B'=\sqrt{fA}B'=\sqrt{fB'}=(b_1\cdots b_l)B'$. Since the above argument does not change with the reordering and the multiplication by units, we may assume $f_i=\prod_{k=j_{i-1}+1}^{j_i} b_j$. Thus $R'\to B'$ factors through
     \[
     R'=k[x_1,\dots,x_n]\to k[y_1,\dots,y_m]\to B',
     \]
     where the first map sends $x_i$ to $\prod_{k=j_{i-1}+1}^{j_i} y_j$. This map is a tensor product of maps of the forms $k[x]\to k[y_1,\dots, y_d]:x\mapsto y_1\cdots y_d$ or $k\to k[y_1,\dots,y_d]$. These maps are flat, so the desired map is flat.\\
     \indent Since flatness is an \'etale local property (\cite[\href{https://stacks.math.columbia.edu/tag/03MM}{03MM}]{stacks-project}), $R'\to A$ is flat.
 \end{proof}
 
 We denote $(A\otimes_RA)^{\mathrm{rep}}_{M\oplus_PM}:=(A\otimes_R A)\otimes_{\mathbb{Z}[M\oplus_PM]}\mathbb{Z}[(M\oplus_PM)^{\mathrm{rep}}]$.

 \begin{lemma}\label{flat base change}
     If maps of pre-log rings $(R,P)\to(A,M)\to(B,N)$ induce flat ring maps $A\to B$ and $(A\otimes_RA)^{\mathrm{rep}}_{M\oplus_PM}\to(B\otimes_RB)^{\mathrm{rep}}_{N\oplus_PN}$, then
     \[ (A\otimes^{\bbL}_{(A\otimes_RA)^{\mathrm{rep}}_{M\oplus_PM}}A)\otimes_AB\simeq (B\otimes_AB)^{\mathrm{rep}}_{N\oplus_MN}\otimes^{\mathbb{L}}_{(B\otimes_R B)^{\mathrm{rep}}_{N\oplus_PN}}B.
     \]
 \end{lemma}

 \begin{proof}
    By assumption, 
    \begin{align*}    (A\otimes^{\bbL}_{(A\otimes_RA)^{\mathrm{rep}}_{M\oplus_PM}}A)\otimes_AB&\simeq (A\otimes^{\bbL}_{(A\otimes_RA)^{\mathrm{rep}}_{M\oplus_PM}}B)\\
    &\simeq (A\otimes_{(A\otimes_RA)^{\mathrm{rep}}_{M\oplus_PM}}(B\otimes_R B)^{\mathrm{rep}}_{N\oplus_PN})\otimes^{\bbL}_{(B\otimes_R B)^{\mathrm{rep}}_{N\oplus_PN}}B\\
    &\simeq (B\otimes_AB)^{\mathrm{rep}}_{N\oplus_MN}\otimes^{\mathbb{L}}_{(B\otimes_R B)^{\mathrm{rep}}_{N\oplus_PN}}B,
    \end{align*}
    where the third equivalence is \cite[Proposition 5.4.]{logHH}.
 \end{proof}

 \begin{prop}[Zariski descent]\label{Zariski descent}
 If $(A,f)$ is a normal crossing modulus pair such that $f=f_1^{r_1}\cdots f_n^{r_n}$ is a product of prime elements, then for any multiplicatively closed subset $S\subset A$, there is an equivalence
     \[     \ulPHH(S^{-1}A,f)\simeq\ulPHH(A,f)\otimes_AS^{-1}A.
     \]
 \end{prop}

 \begin{proof}
     By \cref{comparison2}, 
     \begin{align}
     \ulPHH(A,f)\simeq A\otimes^{\bbL}_{(A\otimes_kA)^\rep_{M\oplus_{k^*}M}}A,
     \end{align}
     same holds for $(S^{-1}A,f)$. To use the \cref{flat base change}, we check that 
     \begin{align}\label{base chanege dondition map}(A\otimes_kA)^\rep_{M\oplus_{k^*}M}\to(S^{-1}A\otimes_kS^{-1}A)^\rep_{N\oplus_{k^*}N}
     \end{align}
     is flat, where $M=A\cap A[f^{-1}]*$, $N=S^{-1}A\cap S^{-1}A[f^{-1}]*$. Assume $f_1,\dots,f_i\in N$, $f_{i+1}\dots,f_n\notin N$. Since $f_1,\dots,f_i\in N$ are primes in $A$, the inclusion 
     \[
     f_1^\bbZ\times\cdots\times f_i^\bbZ\to (S^{-1}A)^*
     \]
     has a section. Thus 
     \[
     N=(S^{-1}A)^*\times f_{i+1}^\bbN\times\cdots\times f_n^\bbN=N_2\times f_1^\bbZ\times\cdots\times f_i^\bbZ\times f_{i+1}^\bbN\times\cdots\times f_n^\bbN
     \]
     for some abelian group $N_2$. Let $N_1\cong \bbZ^i\oplus\bbN^{n-i}$ be the other summand of $N$. Then the map $M\to N$ is a sum of $M_1=\prod_{j=1}^n f_j^\bbN\to\prod_{j=1}^if_j^\bbZ\times\prod_{j=i+1}^nf_j^\bbN=N_1$ and $M_2=A^*\to N_2$. We have
     \begin{align*}         (A\otimes_kA)^\rep_{M\oplus_{k^*}M}&\cong(A\otimes_kA)\otimes_{\bbZ[M_1\oplus M_1]}\bbZ[(M_1\oplus M_1)^\rep]\\     &\to(S^{-1}A\otimes_kS^{-1}A)\otimes_{\bbZ[M_1\oplus M_1]}\bbZ[(M_1\oplus M_1)^\rep]\\
     &\cong (S^{-1}A\otimes_kS^{-1}A)\otimes_{\bbZ[N_1\oplus N_1]}\bbZ[N_1\oplus N_1]\otimes_{\bbZ[M_1\oplus M_1]}\bbZ[(M_1\oplus M_1)^\rep]\\
     &\xrightarrow{\sim}(S^{-1}A\otimes_kS^{-1}A)\otimes_{\bbZ[N_1\oplus N_1]}\bbZ[(N_1\oplus N_1)^\rep]\\
     &\cong (S^{-1}A\otimes_kS^{-1}A)^\rep_{N\oplus_{k^*}N},
     \end{align*}
     where \cref{chart} implies the first and last isomorphisms. The second map is flat because $A\to S^{-1}A$ is flat. Thus to show that \cref{base chanege dondition map} is flat, it is sufficient to show the forth map is an isomorphism.\\
     \indent Since $N_1=\bbZ^i\oplus \bbN^{n-i}$, 
     \begin{align*}
     (N_1\oplus N_1)^\rep=(\bbZ^i\oplus\bbZ^i)^\rep\oplus(\bbN^{n-i}\oplus\bbN^{n-i})^\rep=(\bbZ^i\oplus\bbZ^i)\oplus(\bbN^{n-i}\oplus\bbN^{n-i})^\rep
     \end{align*}
     Since $M_1=\bbN^i\oplus\bbN^{n-i}$,
     \begin{align*}
         &(N_1\oplus N_1)\oplus_{(M_1\oplus M_1)}(M_1\oplus M_1)^\rep\\         \cong&\left((\bbZ^i\oplus\bbZ^i)\oplus_{(\bbN^i\oplus\bbN^i)}(\bbN^i\oplus\bbN^i)^\rep\right)\oplus\left((\bbN^{n-i}\oplus\bbN^{n-i})\oplus_{(\bbN^{n-i}\oplus\bbN^{n-i})}(\bbN^{n-i}\oplus\bbN^{n-i})^\rep\right)\\
         \cong&(\bbZ^i\oplus\bbZ^i)\oplus(\bbN^{n-i}\oplus\bbN^{n-i})^\rep\\
         \cong&(N_1\oplus N_1)^{\rep}.
     \end{align*}
     Since \cref{base chanege dondition map} is flat, by \cref{flat base change}, we have
     \begin{align*}         (A\otimes^{\bbL}_{(A\otimes_kA)^\rep_{M\oplus_{k^*}M}}A)\otimes_AS^{-1}A\simeq(S^{-1}A\otimes_AS^{-1}A)^\rep_{N\oplus_MN}\otimes^\bbL_{(S^{-1}A\otimes_kS^{-1}A)^\rep_{N\oplus_{k^*}N}}S^{-1}A.
     \end{align*}
     However, $S^{-1}A\otimes_AS^{-1}A\cong S^{-1}A$, and
     \begin{align*}
         (N\oplus_MN)^\rep&=((N_1\oplus_{M_1}N_1)\oplus(N_2\oplus_{M_2}N_2))^\rep\\
         &=((\bbZ^i\oplus_{\bbN^i}\bbZ^i)\oplus(\bbN^{n-i}\oplus_{\bbN^{n-i}}\bbN^{n-i})\oplus(N_2\oplus_{M_2}N_2))^\rep\\
         &=(\bbZ^i\oplus_{\bbN^i}\bbZ^i)^{\rep}\oplus(\bbN^{n-i}\oplus_{\bbN^{n-i}}\bbN^{n-i})^{\rep}\oplus(N_2\oplus_{M_2}N_2)^\rep\\
         &=\bbZ^{i}\oplus\bbN^{n-i}\oplus(N_2\oplus_{M_2}N_2)\\
         &=N\oplus_MN.
     \end{align*}
     Hence $(S^{-1}A\otimes_AS^{-1}A)^\rep_{N\oplus_MN}\cong S^{-1}A$ and 
     \begin{align*}(A\otimes^{\bbL}_{(A\otimes_kA)^\rep_{M\oplus_{k^*}M}}A)\otimes_AS^{-1}A&\simeq(S^{-1}A\otimes_AS^{-1}A)^\rep_{N\oplus_MN}\otimes^\bbL_{(S^{-1}A\otimes_kS^{-1}A)^\rep_{N\oplus_{k^*}N}}S^{-1}A\\
     &\simeq S^{-1}A\otimes^\bbL_{(S^{-1}A\otimes_kS^{-1}A)^\rep_{N\oplus_{k^*}N}}S^{-1}A.
     \end{align*}
 \end{proof}

 \begin{prop}[\'Etale descent]\label{etale descent}
     Let $(A,f)$ be a strict normal crossing modulus pair. Suppose that $A\to B$ is an \'etale morphism and $f$ is a product of prime elements in $A$ and $B$. Then
     \[     \ulPHH(B,f)\simeq\ulPHH(A,f)\otimes_AB.
     \]
 \end{prop}

 \begin{proof}
     By \cref{comparison theorem of logHH} and \cref{comparison2}, $\ulPHH(A,f)$ and $\ulPHH(B,f)$ are equivalent to the complexes $A\otimes^{\bbL}_{(A\otimes_kA)^{\mathrm{rep}}_{M\oplus_{k^*}M}}A$ and $B\otimes^{\bbL}_{(B\otimes_kB)^{\mathrm{rep}}_{N\oplus_{k^*}N}}B$, where $M=A\cap A[f^{-1}]^*$, and $N=B\cap B[f^{-1}]^*$. By \cref{Zariski descent}, we may assume that $\varphi :A\to B$ is an \'etale homomorphism of local rings. Let $\mathfrak{m}_A$ be a maximal ideal of $A$ and $x_1,\dots,x_n$ be its regular system of parameters. By \cite[\href{https://stacks.math.columbia.edu/tag/09CC}{09CC}]{stacks-project}, $s_1,\dots,s_n$ is a regular system of parameters in Noetherian local ring $(S,\mathfrak{m})$ if and only if $(s_1,\dots,s_n)=\mathfrak{m}$ and the homology of the Koszul complex $H_i(K_\bullet(s_1,\dots,s_n))=0$ for $i\neq 0$. Since flat local homomorphisms are faithful and since $\mathfrak{m}_AB=\mathfrak{m}_B$, $x_1,\dots,x_n$ is also a regular system of parameters of $B$. Therefore $M=A\cap A[f^{-1}]^*=A^*\times x_1^\bbN\times\cdots\times x_i^\bbN$ and $N=B\cap B[f^{-1}]^*=B^*\times x_1^\bbN\times\cdots\times x_i^\bbN$. Let $L:=x_1^\bbN\times\cdots\times x_i^\bbN$. Thus
     \begin{align*}       (B\otimes_kB)^\rep_{N\oplus_{k^*}N}&\cong(B\otimes_kB)^\rep_{L\oplus L}\\
     &\cong (B\otimes_kB)\otimes_{(A\otimes_kA)}(A\otimes_kA)^\rep_{L\oplus L}\\
     &\cong (B\otimes_kB)\otimes_{(A\otimes_kA)}(A\otimes_kA)^\rep_{M\oplus_{k^*}M}
     \end{align*}
     is flat over $(A\otimes_kA)^\rep_{M\oplus_{k^*}M}$, where the first and third morphisms are isomorphisms by \cref{replete Bar construction} for $n=2$. By \cref{flat base change}, 
     \[     (A\otimes^{\bbL}_{(A\otimes_RA)^{\mathrm{rep}}_{M\oplus_PM}}A)\otimes_AB\simeq (B\otimes_AB)^{\mathrm{rep}}_{N\oplus_MN}\otimes^{\mathbb{L}}_{(B\otimes_R B)^{\mathrm{rep}}_{N\oplus_PN}}B.
     \]
     Since $N\oplus_MN\cong (B^*\oplus_{A^*}B^*)\oplus(L\oplus_L L)\cong(B^*\oplus_{A^*}B^*)\oplus L$, $(N\oplus_MN)^\rep\cong N\oplus_MN$. Therefore $(B\otimes_AB)^{\mathrm{rep}}_{N\oplus_MN}\cong B\otimes_AB\cong B\times C$ because $A\to B$ is \'etale. Hence
     \begin{align*}
         \ulPHH(A,f)\otimes_AB&\simeq (B\times C)\otimes^{\mathbb{L}}_{(B\otimes_k B)^{\mathrm{rep}}}B\\
         &\simeq\ulPHH(B,f)\times (C\otimes^{\mathbb{L}}_{(B\otimes_k B)^{\mathrm{rep}}}B).
     \end{align*}
     For any maximal ideal $\mathfrak{m}$ of $(B\otimes_kB)^{\mathrm{rep}}$, $(C\otimes^{\mathbb{L}}_{(B\otimes_k B)^{\mathrm{rep}}}B)\otimes_{(B\otimes_k B)^{\mathrm{rep}}}(B\otimes_k B)^{\mathrm{rep}}_\mathfrak{m}\simeq C_\mathfrak{m}\otimes^{\mathbb{L}}_{(B\otimes_k B)^{\mathrm{rep}}_\mathfrak{m}}B_{\mathfrak{m}}$. Since $(B\otimes_k B)^{\mathrm{rep}}\to B\otimes_AB\cong B\times C$ is surjective, $B_{\mathfrak{m}}=0$ or $C_{\mathfrak{m}}=0$. Therefore $C\otimes^{\mathbb{L}}_{(B\otimes_k B)^{\mathrm{rep}}}B\simeq0$ and we get an desired isomorphism.
 \end{proof}

 \begin{lemma}
     If $\ulMO(A,f)$ is a flat $A$-module, then the canonical map
     \[     \ulMO(A,f)\otimes_A\ulPHH_q(A,f)\to\ulMHH_q(A,f)
     \]
     is an isomorphism for any $q\geq0$. In particular if $A$ is locally factorial, the above map is an isomorphism.
 \end{lemma}

 \begin{proof}
     Since $\ulMO(A,f)$ is a flat $A$-module, $\ulMO(A,f)\otimes_A\ulPHH_q(A,f)\cong H_q(\ulMO(A,f)\otimes_A\ulPC(A,f))$. Thus it is sufficient to show that the canonical surjection
     \[
     \ulMO(A,f)\otimes\ulPHH(A,f)\to\ulMHH(A,f)
     \]
     of chain complexes is injective. However, there is a  commutative diagram
     \[
     \begin{tikzcd}
         \ulMO(A,f)\otimes_A\ulPHH(A,f) \ar[r] \arrow[d] & \ulMHH(A,f) \ar[d] \\
         \ulMO(A,f)\otimes_A HH(A[f^{-1}]) \ar[r] & HH(A[f^{-1}]).
     \end{tikzcd}
     \]
     The left map is injective because $\ulMO(A,f)$ is flat. The bottom map is an isomorphism by $\ulMO(A,f)\otimes_A A[f^{-1}]\cong A[f^{-1}]$. Therefore the top map is injective.
 \end{proof}
 
 \begin{prop}
     For any modulus pair $(A,f)$, we have isomorphisms $\uline{P}\Omega^1_{(A,f)}\cong\uline{P}HH_1(A,f)$ and $\uline{M}\Omega^1_{(A,f)}\cong\uline{M}HH_1(A,f)$.
 \end{prop}
 \begin{proof}
     Since $b(a_0\otimes a_1)=a_0a_1-a_1a_0=0$,
     \[\uline{P}HH_1(A,f)=\langle \sum c_0\otimes c_1\rangle/\langle c_0c_1\otimes c_2-c_0\otimes c_1c_2+c_2c_0\otimes c_1\rangle,\]
     where $c_i=a_im_i$ for $a_i\in A$, $m_i\in A[f^{-1}]^*$ such that $\prod m_i\in A$. Let $e:\underline{P}C_1(A,f):=\langle \sum c_0\otimes c_1\rangle\to\Omega_{A[f^{-1}]}$ be the map which sends $a_0m_0\otimes a_1m_1\in\underline{P}C_1(A,f)$ to
     \begin{align*}     a_0m_0d(a_1m_1)&=a_0m_0m_1da_1+a_0a_1m_0dm_1\\
     &=a_0m_0m_1da_1+a_0a_1m_0m_1d\log m_1
     \end{align*}
     Since $m_0m_1\in A$ and $d\log m_1=d\log m_1'-d\log m_1''$, where $m_1=\frac{m_1'}{m_1''}$ and $m_1', m_1''\in A\cap A[f^{-1}]^*$, $e$ factors through $\underline{P}\Omega^1_{(A,f)}$. Since $d(c_1c_2)=c_1dc_2+c_2dc_1$ in $\underline{P}\Omega_{(A,f)}^1$, and since
     \[
     b(c_0\otimes c_1 \otimes c_2)=c_0c_1\otimes c_2-c_0\otimes c_1c_2+c_2c_0\otimes c_1,
     \]
     in $\ulPHH(A,f)$, $e$ factors through $\underline{P}HH_1(A,f)$. Therefore $e$ is a map from $\ulPHH_1(A,f)$ to $\ulPOmega^1_{(A,f)}$. Conversely, let $\psi:\underline{P}\Omega^1_{(A,f)}\to\underline{P}HH_1(A,f)$ to be the map which sends $\frac{a}{b}dc$ to $\frac{a}{b}\otimes c$. By definition, $e\psi=\mathrm{id}$ and $\psi e=\mathrm{id}$.\\
     \indent Since $e$ and $\psi$ commutes multiplication by an element of $A[f^{-1}]$ on the left, we can extend the isomorphism to $\ulMOmega^1_{(A,f)}\cong\ulMHH_1(A,f)$.
     \end{proof}

 By the universal property of exterior product, the morphisms $\ulPOmega^0_{(A,f)}=A=\ulPHH_0(A,f)$ and $\ulPOmega^1_{(A,f)}\xrightarrow{\sim}\ulPHH_1(A,f)$ induce the morphism of graded-commutative algebras
 \begin{align}\label{HKR morphism}
 \bigwedge\nolimits^*_A\uline{P}\Omega^1_{(A,f)}\to\uline{P}HH_*(A,f).
 \end{align}

 \begin{thm}[HKR-Theorem]\label{HKR}
     Let $(\mathrm{Spec}A,(f))$ be a strict normal crossing modulus pair such that $f$ is a product of prime elements. Then the composition of \cref{HKR morphism} and the isomorphism constructed in \cite[Lemma 3.12]{KM23b}
     \begin{align}\label{HKR morphism 2}
     \uline{P}\Omega^*(A,f)\cong\bigwedge\nolimits^*_A\uline{P}\Omega^1_{(A,f)}\to\uline{P}HH_*(A,f)
     \end{align}
     is an isomorphism. 
     Furthermore, there is an isomorphism
     \begin{align*}
         \uline{M}\Omega^*(A,f)\cong\uline{M}HH_*(A,f).
     \end{align*}  
 \end{thm}

 \begin{proof}
     First, we prove that \cref{HKR morphism 2} is isomorphism if \[(A,f)=(k[x_1,\dots,x_s,y_1,\dots,y_t],x_1^{r_1}\cdots x_s^{r_s}).\] There is a commutative diagram
     \[
     \begin{tikzcd}
         \ulPOmega^q_{(A,f)} \ar[r] \arrow[d] & \ulPHH_q(A,f) \ar[d] \\
         \Omega^q_{A[f^{-1}]} \ar[r, "\sim"] & HH_q(A[f^{-1}] ),
     \end{tikzcd}
     \]
     where the bottom map is an isomorphism by the HKR-theorem for the smooth $k$-algebra $A[f^{-1}]$. Here, $\ulPOmega^q_{(A,f)}$ is a direct summand of $\Omega^q_{A[f^{-1}]}$ and $\ulPHH_q(A,f)$ is a direct summand of $HH_q(A[f^{-1}])$. Indeed,
     \begin{align*}
         \Omega^q_{A[f^{-1}]}&=\bigoplus_{i_j\in\mathbb{Z},\ k_l\geq0,\ s_m+t_n=q}k\cdot x_1^{i_1}\cdots x_s^{i_s}y_1^{k_1}\cdots y_t^{k_t}dx_{s_1}\wedge\cdots\wedge dx_{s_m}\wedge dy_{t_1}\wedge\cdots\wedge dy_{t_n}\\         &=\ulPOmega^q_{(A,f)}\oplus D,
     \end{align*}
     where
     \begin{align*}     \ulPOmega^q_{(A,f)}=&\bigoplus_{i_j,k_l\geq0,\ s_m+t_n=q}k\cdot x_1^{i_1}\cdots x_s^{i_s}y_1^{k_1}\cdots y_t^{k_t}\cdot d\log x_{s_1}\wedge\cdots\wedge d\log x_{s_m}\wedge dy_{t_1}\wedge\cdots\wedge dy_{t_n}\\
     D^q:=&\bigoplus_{\exists i_{s_m'}<0,\ s_m+t_n=q}k\cdot x_1^{i_1}\cdots x_s^{i_s}y_1^{k_1}\cdots y_t^{k_t}\cdot d\log x_{s_1}\wedge\cdots\wedge d\log x_{s_m}\wedge dy_{t_1}\wedge\cdots\wedge dy_{t_n}.
     \end{align*}
     Also,
     \begin{align*}         
     A[f^{-1}]^{\otimes q+1}&=\bigoplus_{i_j\in\mathbb{Z}^s,\  k_l\in(\mathbb{Z}_{\geq0})^t} k\cdot x^{i_0}y^{k_0}\otimes\cdots\otimes x^{i_q}y^{k_q}\\
         &=\ulPC_q(A,f)\oplus E_q
     \end{align*}
     where $x^{i_j}=x_1^{i_{j,1}}\cdots x_s^{i_{j,s}}$, $y^{k_l}=y_1^{k_{l,1}}\cdots y_t^{k_{l,t}}$, and
     \begin{align*}
         \ulPC_q(A,f)&:=\bigoplus_{x^{i_0}y^{k_0}\cdots x^{i_q}y^{k_q}\in A} k\cdot x^{i_0}y^{k_0}\otimes\cdots\otimes x^{i_q}y^{k_q}\\
         E_q&:=\bigoplus_{x^{i_0}y^{k_0}\cdots x^{i_q}y^{k_q}\notin A} k\cdot x^{i_0}y^{k_0}\otimes\cdots\otimes x^{i_q}y^{k_q}.
     \end{align*}
     By the definition of the differential on $HH(A[f^{-1}])$, it decompose two chain complexes $HH(A[f^{-1}])=\ulPHH(A,f)\oplus E_*$, moreover $HH_q(A[f^{-1}])=\ulPHH_q(A,f)\oplus H_q(E)$. Since the canonical map $\Omega^q_{A[f^{-1}]}\xrightarrow{\sim}HH_q(A[f^{-1}])$ sends $\ulPOmega^q_{(A,f)}$ to $\ulPHH_q(A,f)$ and $D^q$ to $E_q$, the canonical map $\ulPOmega^q_{(A,f)}\to\ulPHH_q(A,f)$ is an isomorphism.\\
     \indent For strict normal crossing modulus pair $(\mathrm{Spec}A,(f))$, there is an open covering $\{\mathrm{Spec}A_k\}$ of $\mathrm{Spec}(A)$, such that there is an \'etale morphism $\mathrm{Spec}A_k\to\mathrm{Spec}(k[x_1,\dots,x_n])$ with $(f)\times_{\mathrm{Spec}A}\mathrm{Spec}A_k=(x_1^{r_1}\cdots x_i^{r_i})\times_{\mathrm{Spec}(k[x_1,\dots,x_n])}\mathrm{Spec}A_k$ and  $x_j$ primes in $A_k$ by \cref{local coodinate}. Therefore \'etale descent of $\ulPHH_q$ and $\ulPOmega^q$ (\cite[Proposition 3.17]{KM23b}) induces the HKR isomorphism  for $(\mathrm{Spec}A,(f))$.\\
     \indent Since $\ulMOmega^q_{(A,f)}\cong\ulMO(A,f)\otimes_A\ulPOmega^q_{(A,f)}$ and $\ulMHH_q(A,f)\cong\ulMO(A,f)\otimes_A\ulPHH_q(A,f)$, the canonical map is an isomorphism.
 \end{proof}

 \section{Modulus Cyclic Homology}
 In this section, we will construct  modulus (periodic and negative) cyclic homology. From now on, we assume characteristic of $k$ is $0$.

 Before doing that, recall the cyclic homology of a flat $k$-algebra $A$. The cyclic group $C_{n+1}$ acts on $A^{\otimes n+1}$ by
 \[
 t_n(a_0\otimes\cdots\otimes a_{n-1}\otimes a_n)=a_n\otimes a_1\otimes\cdots\otimes a_{n-1}.
 \]
 This action and simplicial structure on $HH(A)$ is compatible, that is, $HH(A)$ is a functor from the cyclic category $\Lambda^{\mathrm{op}}$ to the derived $\infty$-category $\calD(k)$ of chain complexes of $k$-modules (cf.\cite{Lod98}). Its geometric realization 
 \[ HH(A)=\varinjlim_{[n]\in\Delta^{\mathrm{op}}}A^{\otimes n+1}\in \mathrm{Fun}(B\bbT,\calD(k))
 \]
 has an action by the circle $\bbT$. Cyclic homology and negative cylcic homology are defined by
 \begin{align*}
     HC(A)&:=(HH(A))_{h\bbT}\\
     HC^-(A)&:=(HH(A))^{h\bbT}
 \end{align*}
 the homotopy orbits and homotpy fixed points. For any $X\in \mathrm{Fun}(B\bbT,\calD(k))$, there is a norm map $N:\Sigma X_{h\bbT}\to X^{h\bbT}$ (cf. \cite{NikSch18}), whose cofiber is called \emph{Tate cohomology} and denoted by $X^{t\bbT}$. Periodic cyclic homology $HP(A)$ of $A$ is the Tate cohomology $(HH(A))^{t\bbT}$.\\
 \indent For a cyclic object $M_\bullet\in\mathrm{Fun}(\Lambda^{\mathrm{op}},\calD(k))$, we can define operators
 \begin{align*}
  b&:=\sum_{i=0}^n(-1)^id_i:M_n\to M_{n-1}\\
  B&:=(1-(-1)^{n+1}t_{n+1})t_{n+1}s_n\left(\sum_{i=0}^n(-1)^{ni}t_n^i\right)
:M_n\to M_{n+1},
 \end{align*}
 and these maps satisfies $b^2=B^2=bB+Bb=0$. If $M=HH(A)$, there is a commutative diagram 
\begin{equation}\label{commutativity of operators}
 \begin{tikzcd}
     A^{\otimes n} \ar[d, "e"] & A^{\otimes n+1} \ar[r, "B"] \ar[l, "b"']  \ar[d, "e"] & A^{\otimes n+2} \ar[d, "e"]\\
     \Omega^{n-1}_A & \Omega^n_A  \ar[l, "0"'] \ar[r, "d"] & \Omega^{n+1}_A,
 \end{tikzcd}
 \end{equation}
 where the vertical morphisms $e:A^{\otimes n+1}\to \Omega^n_A$ are
 \begin{align}\label{HKR inverse}
 e(a_0\otimes\cdots\otimes a_n)=\frac{1}{n!}a_0da_1\wedge\cdots\wedge da_n.
 \end{align}
 By \cite{Hoy18}, $HP(A)$ is equivalent to the total complex of the bicomplex

 \begin{equation}\label{bicomplex}
 \begin{tikzcd}
      & \vdots \ar[d] & \vdots \ar[d] & \vdots \ar[d] \\
     \cdots & A^{\otimes 3} \ar[l] \ar[d, "b"] & A^{\otimes 2} \ar[l, "B"'] \ar[d, "b"] & A \ar[l, "B"']\\
     \cdots & A^{\otimes 2} \ar[l] \ar[d, "b"] & A \ar[l, "B"']\\
     \cdots & A \ar[l] .
 \end{tikzcd}
 \end{equation}
 Similarly, if we remove all negatively (resp. positively) graded columns, then its total complex is equivalent to $HC(A)$ (resp. $HC^-(A)$).

 \indent If $f=f_1^{r_1}\cdots f_i^{r_i}$ is a product of prime elements, then $C_{n+1}$ acts on the $n$-th group of $\ulMHH(A,f)$. Indeed, since $\ulMO(A,f)=\frac{1}{f_1^{r_1-1}\cdots f_i^{r_i-1}}\cdot A$, we have
 \begin{align}\label{circle action}
 \begin{split}
 &t\left(\frac{a_0}{f_1^{r_1-1}\cdots f_i^{r_i-1}}b_0\otimes\cdots\otimes a_nb_n\right)
 =a_nb_n\otimes \frac{a_0b_0}{f_1^{r_1-1}\cdots f_i^{r_i-1}}\otimes\cdots\otimes a_{n-1}b_{n-1}\\
 =&\frac{a_n}{f_1^{r_1-1}\cdots f_i^{r_i-1}}(f_1^{r_1-1}\cdots f_i^{r_i-1}b_n)\otimes a_0\frac{b_0}{f_1^{r_1-1}\cdots f_i^{r_i-1}}\otimes\cdots\otimes a_{n-1}b_{n-1}\\
 \in&\ulMHH(A,f).
 \end{split}
 \end{align}
 Thus $\ulMHH(A,f)$ also has a circle action.
 \begin{dfn}
     For a modulus pair $(A,f)$ such that $f=f_1^{r_1}\cdots f_i^{r_i}$ is a product of prime elements, define
     \begin{align*}
         &\ulMHC(A,f):=(MHH(A,f))_{h\bbT}\\
         &\ulMHNC(A,f):=(MHH(A,f))^{h\bbT}\\
         &\ulMHP(A,f):=(MHH(A,f))^{t\bbT}.
     \end{align*}
     We let $\ulMHC_n(A,f)$, $\ulMHNC_n(A,f)$, and $\ulMHP_n(A,f)$ be their homologies, and call them \emph{modulus (negative and periodic) cyclic homologies}.
 \end{dfn}

 \begin{thm}\label{isom for cyclic}
    If $(A,f)$ is strict normal crossing and $f=f_1^{r_1}\cdots f_n^{r_n}$ for prime elements $f_j\in A$, then
    \begin{align*}
        &\ulMHC_n(A,f)\cong (\ulMOmega^n_{(A,f)}/d\ulMOmega^{n-1}_{(A,f)})\times H^{n-2}(\ulMOmega^*_{(A,f)})\times H^{n-4}(\ulMOmega^*_{(A,f)})\times\cdots\\        &\ulMHNC(A,f)\cong\mathrm{Ker}(\ulMOmega_{(A,f)}^n\to \ulMOmega_{(A,f)}^{n+1})\times H^{n+2}(\ulMOmega_{(A,f)^*})\times H^{n+4}(\ulMOmega_{(A,f)}^*)\times\cdots\\
        &\ulMHP_n(A,f)\cong \prod_{p\in\mathbb{Z}}H^{2p-n}(\ulMOmega^*(A,f)).
    \end{align*}
\end{thm}

\begin{proof}
    We can restrict the $e:HH(A[f^{-1}])\to\prod_{p\geq0}\Omega^p_{A[f^{-1}]}[p]$ from \cref{HKR inverse} to subcomplexes $e:\ulMHH(A,f)\to \prod_{p\geq0}\ulMOmega^p_{(A,f)}[p]$, because
    \begin{align*}
    e\left(\frac{a}{f}\cdot a_0b_0\otimes a_1b_1\otimes\cdots\otimes a_nb_n\right)&:=\frac{a}{n!f}\cdot a_0b_0d(a_1b_1)\wedge\cdots\wedge d(a_nb_n)\\
    &=\frac{a}{n!f}\cdot a_0b_0\cdots b_n\sum\omega_1\wedge\cdots\wedge\omega_n\in\ulMOmega^n_{(A,f)},
    \end{align*}
    where $\omega_i$ is $da_i$ or $a_id\log b_i$. The morphism $e$ is the inverse map of the HKR isomorphism $\ulMOmega^n_{(A,f)}\xrightarrow{\sim}\ulMHH_n(A,f)$ (cf. \cite[section 9.8]{Wei94}). Thus $e$ is a quasi-isomorphism. Similar for periodic cyclic homology $HP(A)$ \cref{bicomplex}, by \cite{Hoy18}, modulus periodic cyclic homologies are equivalent to the total complex $\mathrm{Tot}^{\prod}(\ulMHH(A,f)[\bullet],B)$, where $(\ulMHH(A,f)_\bullet,B)$ is the bicomplex associated to the chain complex of complexes
    \[
    (\cdots\to\ulMHH(A,f)[n]\xrightarrow{B}\ulMHH(A,f)[n-1]\to\cdots).
    \]
    By the commutativity of \cref{commutativity of operators}, the equivalence $e$ commutes operators $B:\ulMHH(A,f)\to\ulMHH(A,f)[-1]$ and the de Rham differential $d:\prod_{p\geq0}\ulMOmega^p_{(A,f)}[p]\to(\prod_{p\geq0}\ulMOmega^p_{(A,f)}[p])[-1]$, so there are equivalences
    \[
    \mathrm{Tot}^{\prod}(\ulMHH(A,f)[\bullet],B)\xrightarrow[e]{\sim}\mathrm{Tot}^{\prod}(\prod_{p\geq0}\ulMOmega^p_{(A,f)}[\bullet],d)\cong\prod_{p\in\mathbb{Z}}(\ulMOmega_{(A,f)}^\bullet)^{-*}[2p].
    \]
    Restricting bicomplexes, we obtain equivalences
    \begin{align*}       &\ulMHC(A,f)\simeq\prod_{p\in\mathbb{Z}}(\ulMOmega_{(A,f)}^{\leq p})^{-*}[2p]\\  &\ulMHNC(A,f)\simeq\prod_{p\in\mathbb{Z}}(\ulMOmega_{(A,f)}^{\geq p})^{-*}[2p]\\     &\ulMHP(A,f)\simeq\prod_{p\in\mathbb{Z}}(\ulMOmega_{(A,f)}^\bullet)^{-*}[2p].
    \end{align*}
    If we take homologies, then we obtain the desired isomorphisms
    \begin{align*}       &\ulMHC_n(A,f)\cong\prod_{p\in\mathbb{Z}}H^{2p-n}(\ulMOmega_{(A,f)}^{\leq p})\\  &\ulMHNC_n(A,f)\cong\prod_{p\in\mathbb{Z}}H^{2p-n}(\ulMOmega_{(A,f)}^{\geq p})\\     &\ulMHP_n(A,f)\cong\prod_{p\in\mathbb{Z}}H^{2p-n}(\ulMOmega_{(A,f)}^\bullet).
    \end{align*}
\end{proof}

 \section{Realizations}

 We will construct objects of $\underline{\mathbf{M}}\mathbf{DM}^{\mathrm{eff}}_k$ which represents the modulus Hochschild homology and modulus (periodic and negative) cyclic homologies.

 \begin{dfn}
     For a modulus pair $\calX:=(\overX,X^\infty)$, let $\calM_\calX:=\calO_{\overX}\cap\calO_X^*$, where $X:=\overX\setminus X^\infty$.  For a map of sheaf of rings $\calA\to\calB$, let $B^\cyc_{\calA}(\calB)$ be the cyclic Bar construction in the category $Sh(\overX_{\Zar},\mathrm{Ring})$ of sheaves of rings on $\overX_{\Zar}$. Let $B^\rep_{\bbZ[k^*]}(\bbZ[\calM_{\calX}])$ be the subsheaf of simplicial rings $B^\cyc_{\bbZ[k^*]}(\bbZ[\calM_\calX^{\mathrm{gp}}])$ such that its $n$-th simplex is the monoid ring of pullback of monoids
     \[
     \begin{tikzcd}
        (\oplus_{k^*}^{n+1}\calM_{\calX})^\rep \ar[r] \ar[d] & (\oplus_{k^*}^{n+1}\calM_{\calX})^{\mathrm{gp}} \ar[d] \\
        \calM_{\calX} \ar[r] & \calM_{\calX}^{\mathrm{gp}}.
     \end{tikzcd}
     \]
     Let $\calulPHH(\calX)$ be the sheaf of simplicial rings $B^\cyc_{k}(\calO_{\overX})\otimes_{B^\cyc_{\bbZ[k^*]}(\bbZ[\calM_\calX])}B^\rep_{\bbZ[k^*]}(\bbZ[\calM_{\calX}])$. Define $\calulMHH(\mathcal{X})$ to be the simplicial $\calO_{\overX}$-module $\ulMO(\mathcal{X})\otimes_{\mathcal{O}_{\overline{X}}}\calulPHH(\mathcal{X})$.
 \end{dfn}

 For a strict normal crossing modulus pair $\calX$, $\calulMHH(\calX)$ has a circle action Zariski locally by \cref{circle action}. By \cref{lem:factorization}, we can define a circle action on $\calulMHH(\calX)$ globally, that is, $\calulMHH(\calX)\in\mathrm{Fun}(\Lambda^\op,\calD(\overX))$ is a cyclic object of $\calD(\overX)$.

 \begin{dfn}
     For strict normal crossing modulus pair $\calX$, \emph{modulus Hochschild homology} $\ulMHH(\calX)\in\mathrm{Fun}(B\bbT,\calD(k))$ of $\calX$ is an image of the composition of functors\footnote{We use homological conventions for $\calD(k)$. So the objects are chain complexes, not cochain complexes.}
     \[
     \mathrm{Fun}(\Lambda^\op,\calD(\overX))\xrightarrow{|-|}\mathrm{Fun}(B\bbT,\calD(\overX))\xrightarrow{R\Gamma(\overX,-)^{-*}}\mathrm{Fun}(B\bbT,\calD(k)).
     \]
     The \emph{modulus (negative and periodic) cyclic homology} of $\calX$ are 
     \begin{align*}
         &\ulMHC(\calX):=(MHH(\calX))_{h\bbT}\\
         &\ulMHNC(\calX):=(MHH(\calX))^{h\bbT}\\
         &\ulMHP(\calX):=(MHH(\calX))^{t\bbT}.
     \end{align*}
     Let $\ulMHH_n(\calX)$, $\ulMHC_n(\calX)$, $\ulMHNC_n(\calX)$, and $\ulMHP_n(\calX)$ be their homologies.
 \end{dfn}

 \begin{thm}\label{realization theorem}
     For every strict normal crossing modulus pair $\mathcal{X}=(\overline{X},X^\infty)$, we have the following isomorphisms\footnote{We use cohomological conventions for $\ulMDMeff_k$. So the objects are cochain complexes, not chain conmplexes}
     \begin{align*}
     &\mathrm{Hom}_{\ulMDMeff_k}(M(\mathcal{X}),\prod_{p\in\mathbb{Z}}\bfulMOmega^p[p-n])\cong\prod_{p\in\mathbb{Z}}H^{p-n}_{\Zar}(\overX,\ulMOmega^p_{\calX})\cong\ulMHH_n(\mathcal{X})\\
     &\mathrm{Hom}_{\ulMDMeff_k}(M(\mathcal{X}),\prod_{p\in\mathbb{Z}}\bfulMOmega^{\leq p}[2p-n])\cong\prod_{p\in\bbZ}\bbH^{2p-n}_{\Zar}(\overX,\ulMOmega^{\leq p}_{\calX})\cong\ulMHC_n(\mathcal{X})\\
     &\mathrm{Hom}_{\ulMDMeff_k}(M(\mathcal{X}),\prod_{p\in\mathbb{Z}}\bfulMOmega^{\geq p}[2p-n])\cong\prod_{p\in\bbZ}\bbH^{2p-n}_{\Zar}(\overX,\ulMOmega^{\geq p}_{\calX})\cong\ulMHNC_n(\mathcal{X})\\
     &\mathrm{Hom}_{\ulMDMeff_k}(M(\mathcal{X}),\prod_{p\in\mathbb{Z}}\bfulMOmega^\bullet[2p-n])\cong\prod_{p\in\bbZ}\bbH^{2p-n}_{\Zar}(\overX,\ulMOmega^{\bullet}_{\calX})\cong\ulMHP_n(\mathcal{X}).
     \end{align*}
 \end{thm}

 \begin{proof}
     The first isomorphisms are representability of these cohomologies described in \cref{Hodge realization} and \cref{de Rham realization}. The map $e:\calulMHH(\calX)\to\prod_{p\geq0}\ulMOmega^p_{\calX}[p]:a_0\otimes\cdots\otimes a_n\mapsto a_0da_1\wedge\cdots\wedge da_n$ is an equivalence of the chain complexes of sheaves. So we obtain the second isomorphism for $\ulMHH(\calX)$. As observed in the proof of \cref{isom for cyclic}, $\ulMHP(\calX)$ is equivalent to 
     \[
     \mathrm{Tot}^{\prod}(\ulMHH(\calX)[\bullet],B)\xrightarrow[e]{\sim}\mathrm{Tot}^{\prod}((\prod_{p\geq0}R\Gamma(\overX,\ulMOmega^p_{\calX})[p])[\bullet],R\Gamma(\overX,d)),
     \]
     but right hand side is $\prod_{p\in\bbZ}R\Gamma(\overX,\ulMOmega^{*}_{\calX})^{-*}[2p]$. Since cyclic (negative cyclic) homology is a total complex of non-negative (non-positive) parts of bicomplex of periodic cyclic homology, we have equivalences
     \begin{align*}
         &\ulMHC(\calX)\simeq \prod_{p\in\bbZ}R\Gamma(\overX,\ulMOmega^{\leq p}_{\calX})^{-*}[2p]\\
         &\ulMHNC(\calX)\simeq \prod_{p\in\bbZ}R\Gamma(\overX,\ulMOmega^{\geq p}_{\calX})^{-*}[2p]\\
         &\ulMHP(\calX)\simeq \prod_{p\in\bbZ}R\Gamma(\overX,\ulMOmega^{\bullet}_{\calX})^{-*}[2p]
     \end{align*}
     If we take homologies, we get the desired isomorphisms.
 \end{proof}

 \bibliographystyle{amsalpha} 
 \bibliography{references}

\end{document}